\documentclass{amsart}
\usepackage{amsmath,amsfonts,amssymb}
\usepackage{amsthm}
\usepackage{ifthen}
\usepackage{longtable}
\usepackage{array}
\usepackage{url}
\usepackage{enumerate}

\newcommand{\Z}{\mathbb{Z}}
\newcommand{\Q}{\mathbb{Q}}
\newcommand{\R}{\mathbb{R}}
\newcommand{\C}{\mathbb{C}}

\newcommand{\D}{\mathcal{D}}

\newcommand{\ordp}{\textrm{ord}_p}

\newcommand{\p}{\mathfrak{p}}

\newcommand{\LL}{\mathcal{L}}

\newcommand{\sd}{\sqrt{\Delta}}
\newcommand{\ls}{\log_{*}}
\newcommand{\la}{\log{|\alpha|}}

\newcommand{\al}{\alpha}
\newcommand{\be}{\beta}
\newcommand{\ga}{\gamma}

\DeclareMathOperator{\Norm}{N}

\newtheorem*{theorem*}{Theorem}
\newtheorem{theorem}{Theorem}
\newtheorem{lemma}{Lemma}
\newtheorem{corollary}{Corollary}
\newtheorem{proposition}{Proposition}

\theoremstyle{remark}

\title[Diophantine equations of the form $u_n+u_m=w p_1^{z_1} \cdots p_s^{z_s}$]{Effective Resolution of
Diophantine equations of the form $u_n+u_m=w p_1^{z_1} \cdots p_s^{z_s}$}

\author[I. Pink]{Istv\'an Pink}

\author[V. Ziegler]{Volker Ziegler}

\thanks{The research was granted by the Austrian science found (FWF) under the project P 24801-N26.}



\address{I. Pink \newline
         \indent Institute of Mathematics, University of Debrecen \newline
         \indent H-4010 Debrecen, P.O. Box 12, Hungary \newline
         \indent and \newline
         \indent University of Salzburg \newline
         \indent Hellbrunnerstrasse 34/I \newline
         \indent A-5020 Salzburg, Austria}

\email{pinki\char'100science.unideb.hu; istvan.pink\char'100sbg.ac.at}

\address{V. Ziegler \newline
         \indent University of Salzburg \newline
         \indent Hellbrunnerstrasse 34/I \newline
         \indent A-5020 Salzburg, Austria}

\email{volker.ziegler\char'100sbg.ac.at}

\begin{document}

\begin{abstract}
 Let $u_n$ be a fixed non-degenerate binary recurrence sequence with positive discriminant, $w$ a fixed
non-zero integer and $p_1,p_2,\dots,p_s$ fixed, distinct prime numbers.
 In this paper we consider the Diophantine equation $u_n+u_m=w p_1^{z_1} \cdots p_s^{z_s}$ and prove under
mild technical restrictions effective finiteness results. In particular
 we give explicit upper bounds for $n,m$ and $z_1, \dots, z_s$. Furthermore, we provide a rather efficient
algorithm to solve Diophantine equations of the described type and
 we demonstrate our method by an example.
\end{abstract}

 \maketitle

 \section{Introduction}\label{Sec:Intro}

 Recently several authors considered the problem to find all indices $n$ and $m$ and exponents $z$ such that
 $$u_n+u_m=2^z,$$
 where $u_n$ is a fixed recurrence. In particular Bravo and Luca considered the cases, where $u_n$ is the
Fibonacci sequence \cite{Bravo:2016}
 and the Lucas sequence \cite{Bravo:2014} respectively. Also the case where $u_n$ is the generalized
$k$-Fibonacci sequence has been considered
 independently by Bravo, G\'omez and Luca \cite{Bravo:2016a} and Marques \cite{Marques:2016}. In \cite{Bhpr}
Bert\'ok, Hajdu, Pink and R\'abai solved completely equations of the form $u_n=2^a+3^b+5^c$, where $u_n$ is one
of the Fibonacci, Lucas, Pell and associated Pell sequences, respectively.   

 In the present paper we aim to generalize the results due to Bravo and Luca \cite{Bravo:2014,Bravo:2016} and
consider the more general Diophantine equation
\begin{equation} \label{eq:main}
u_n+u_m=w p_1^{z_1} \cdots p_s^{z_s}
\end{equation}
in non-negative integer unknowns $n,m,z_1,\dots,z_s$, where $\{u_n\}_{n\geq 0}$ is a binary non-degenerate
recurrence sequence, $p_1,\dots,p_s$ are distinct primes and
$w$ is a non-zero integer with $p_i\nmid w$ for $1\leq i \leq s$. For reasons of symmetry we assume that
$n\geq m$. Although the ideas of Bravo, G\'omez, Luca and Marques
\cite{Bravo:2014,Bravo:2016,Bravo:2016a,Marques:2016}
can be easily extended to this case, we provide two new aspects treating the title Diophantine equation.

Firstly, the main argument due to Bravo et.al. is to consider two
linear forms in (complex) logarithms and obtain by a clever trick upper bounds for $n,m$ and $z$. In our
approach we replace one of the two linear forms in complex logarithms by linear forms in $p_i$-adic
logarithms for every prime $p_i$ with $1\leq i \leq s$. The advantage in doing this is that instead of
considering a linear form in $s+2$ logarithms we only have to consider a linear form in two $p$-adic
logarithms which reduces the upper bound for $n,m,z_1,\dots,z_s$ drastically for large $s$.

The second novelty in treating this kind of problems is a consequence to our $p$-adic
approach by using a $p$-adic reduction method due to Peth\H{o} and de Weger \cite{Pethoe:1986} instead of real (and complex)
approximation lattices or the method of Baker and Davenport \cite{Baker:1969}. This
approach leads to a better performance of the reduction step and a rather efficient algorithm.

Before we state our main result, let us fix some notations. We call the sequence $\{u_n\}_{n\geq
0}=\{u_n(A,B,u_0,u_1)\}_{n\ge0}$ a binary linear recurrence sequence defined over the integers
if the relation
\begin{equation} \label{eq:bin-rec}
u_n=Au_{n-1}+Bu_{n-2} \quad (n \ge 2)
\end{equation}
holds, where $A,B\in\Z$ with $AB \neq 0$ and $u_0,u_1$ are fixed rational integers with $|u_0| + |u_1|>0$.
The polynomial $f(x)=x^2-Ax-B$ attached to recurrence \eqref{eq:bin-rec} is the so-called companion polynomial of the
sequence $\{u_n\}_{n\ge0}$ and we denote
by $\Delta=A^2+4B$ the discriminant of $f$. Let $\alpha$ and $\beta$ be the roots of the companion polynomial
$f$ and assume that $\Delta\neq 0$, then it is well known that
there exist constants $a=u_1-u_0\beta$ and $b=u_1-u_0\alpha$ such that the following formula holds
\begin{equation} \label{eq:binet}
u_n=\frac{a\alpha^n-b\beta^n}{\alpha-\beta}.
\end{equation}
The sequence $\{u_n\}_{n\ge0}$ is called {\it non-degenerate}, if $ab\alpha\beta \neq 0$ and  $\alpha/\beta$
is not a root of unity.

Throughout the paper we will assume that $u_n$ is non-degenerate and that $\Delta>0$. The last assumption
implies that the sequence $\{u_n\}_{n\geq 0}$
possesses a dominant root, which means that without loss of generality $|\alpha|>|\beta|$. Under these
assumptions and notations the main result
of our paper is now as follows:

\begin{theorem} \label{thm1}
Let $\{u_n\}_{n\geq 0}$ be a non-degenerate binary recurrence with $\Delta>0$. Let us assume that $p_i\nmid
\gcd(A,B)$ for all $1\leq i \leq s$ and furthermore let us assume that
none of the following two conditions hold
\begin{itemize}
\item $\beta=\pm 1$ and $m=\frac{\log (\beta^m 2b/a)}{\log \alpha}$.
\item $\beta=-1$ and there exists an positive odd integer $x$ and integers $t_1,\dots,t_s$ such that
\begin{equation}\label{eq:exceptional}
\frac{w(\alpha+1)}{a(\alpha^x+1)}=p_1^{-t_1}\dots p_s^{-t_s}.
\end{equation}
\end{itemize}
Then there exists an effectively computable constant $C$ depending only on $\{u_n\}_{n\geq 0}$, $w$, $s$,
$p_1,\dots,p_s$ such that
all solutions $(n,m,z_1,\dots,z_s)$ to equation \eqref{eq:main} satisfy
$$\max\{n,m,z_1,\dots,z_s\} < C.$$
\end{theorem}

Let us stress out that during the course of proof of Theorem \ref{thm1} we give a very precise description of how
to compute this bound $C$
(see for instance Proposition \ref{prop:explicit} and Section \ref{Sec:Appendix}). Moreover, we present an
easy implementable algorithm
to solve Diophantine equation~\eqref{eq:main} completely (see Section \ref{Sec:Algorithm}) under the assumptions of Theorem
\ref{thm1}.

Let us discuss the technical assumptions made in Theorem \ref{thm1}.
First, the assumption that $p_i\nmid \gcd(A,B)$ for all $1\leq i \leq s$ is to avoid technical difficulties. Using instead of
lower bounds for linear forms of $p$-adic logarithms lower bounds for linear forms of complex logarithms in Section
\ref{Sec:p-adic} one can avoid these difficulties, e.g. by adopting the method due to Bravo and Luca \cite{Bravo:2016}.
However we do not intend to further discuss this case in detail.

In the case that the first exceptional case of Theorem \ref{thm1} holds, i.e. that $\beta=\pm
1$ and $m=\frac{\log (\beta^m 2b/a)}{\log \alpha}$, equation \eqref{eq:main} can be rewritten as
$$\frac{a\alpha^n}{\alpha \pm 1}=w p_1^{z_1}\dots p_s^{z_s}.$$
It is easy to see that this Diophantine equation may have infinitely many solutions which are easy to determine. In particular,
it is possible to determine all solutions to such an equation by slightly adapting our method.
Note that $\alpha$ has to be a rational integer if $\beta=\pm 1$.
For instance, let us choose
$a=b=\beta=w=s=1$ and $\alpha=p_1=2$. Then equation \eqref{eq:main} turns into
$$ \frac{2^n-1}{2-1}+\frac{2^m-1}{2-1}=2^{z_1}$$
which has infinitely many solutions of the form $m=1$ and $n=z_1$.

Let us turn to the second exceptional case of Theorem \ref{thm1}. If $\beta=-1$ and $n-m$ is odd and let us
write $x=n-m>0$, then Diophantine equation \eqref{eq:main} turns into
$$\frac{a\alpha^n+b(-1)^n}{\alpha+1}+\frac{a\alpha^m+b(-1)^m}{\alpha+1}=\frac{a\alpha^m(\alpha^x+1)}{\alpha+1}
=w p_1^{z_1}\dots p_s^{z_s}.$$
If we insert now the assumed relation \eqref{eq:exceptional} we get
$$ \alpha^m=p_1^{z_1-t_1}\dots p_s^{z_s-t_s}$$
and it is easy to see to see that this Diophantine equation may have infinitely many solutions. Moreover, also in this case it is easy
to determine all solutions to such an equation.
In view of Theorem \ref{thm1} only the question how to find all solutions to equation \eqref{eq:exceptional} remains to
solve Diophantine equation \eqref{eq:main} completely in any case.
However, Diophantine equation \eqref{eq:exceptional} is nothing else than
$$u_x=w'p_1^{t_1}\dots p_s^{t_s}$$
with $w'=w/a$. This type of Diophantine equation has been studied by Peth\H{o} and de Weger \cite{Pethoe:1986}
and they gave a practical algorithm how to solve such equations.
Moreover, we will also discuss this kind of Diophantine equation in Section \ref{Sec:n=m}. Let us
provide an example with infinitely many solutions in this case.
For instance let $a=b=w=1$, $\beta=-1$ and $\alpha=p_1\dots p_s$. Then it is easy to see that equation \eqref{eq:main} has
infinitely many solutions of the form $n=m+1$ and $z_1=\dots=z_s=m$.

Finally let us outline the plan of the paper.
The next three sections will provide a proof of Theorem \ref{thm1}. First, we introduce some further notations
and prove some auxiliary results in Section
\ref{Sec:notation}. In Section \ref{Sec:p-adic} we use lower bounds of linear forms in two $p$-adic logarithms
due to Bugeaud and Laurent \cite{Bugeaud:1996} in order to prove some kind of gap principle.
In the Sections \ref{Sec:n=m} and \ref{Sec:n>m} we exploit ideas due to Peth\H{o} and de Weger
\cite{Pethoe:1986} and Bravo and Luca \cite{Bravo:2014,Bravo:2016} respectively.
This leads us to our upper bound $C$ in Theorem \ref{thm1}.
Due to the use of Baker's method the constant $C$ in Theorem~\ref{thm1} is usually very large. By using the
LLL-algorithm and
ideas due to de Weger \cite{deWeger:1987} and Peth\H{o} and de Weger \cite{Pethoe:1986} it is possible to reduce these bounds considerably in concrete
examples. In Section~\ref{Sec:Reduction} we will discuss the method of
de Weger \cite{deWeger:1987} and Peth\H{o} and de Weger \cite{Pethoe:1986} and show how to apply them to our problem. All together this provides an algorithm to solve
Diophantine equations of type \eqref{eq:main} completely
(see Section \ref{Sec:Algorithm}). In order to demonstrate the efficiency of our algorithm we solve the
Diophantine equation
$$u_n+u_m=2^{z_1}3^{z_2}\cdots 199^{z_{46}}$$
completely, where $u_n$ is either the Fibonacci sequence or the Lucas sequence (see Section
\ref{Sec:Example}). Since it is hard to track all the constants
which appear in the proof of Theorem \ref{thm1}, we provide in the final section a list of all constants and
their explicit determination.

 \section{Notations and auxiliary results}\label{Sec:notation}

In this section we keep the notations of the introductory section. However, before we start with the proof of
Theorem \ref{thm1} we need to introduce some more notations.
For a positive real number $x>0$ we define the function $\ls$ by
$$
\ls:\R_{>0}\rightarrow \R  \quad \ls{x}:=\max\{0 ,\log{x}\}.
$$
By $\varphi=\frac{1+\sqrt{5}}{2}$ we denote the golden ratio. Finally we write
$K=\Q(\alpha,\beta)=\Q(\sqrt{\Delta})$ for the number field corresponding
to our binary sequence $\{u_n\}_{n\geq 0}$ and define $d_K=[K:\Q]$.

In the rest of the paper there will appear constants $c_1,c_2,\cdots$ and also occasionally constants of
slightly different form, e.g. $c_{8,i}$, which are
all explicitly computable. Sometimes we do not state them explicitly in our results or proofs for aesthetic
reasons. However, Section \ref{Sec:Appendix} provides a list of all
constants and their explicit determination. We advise readers who wish to keep track of all constants and
their explicit determination to keep a bookmark at Section \ref{Sec:Appendix}.

The first lemma is an elementary result due to Peth{\H o} and de Weger \cite{Pethoe:1986}. It will be used
over and over again
in the proof of Theorem \ref{thm1}. For a proof of Lemma \ref{lem:pdw} we refer to \cite[Appendix
B]{Smart:DiGL}.

\begin{lemma} \label{lem:pdw}
Let $u,v \geq 0, h \geq 1$ and $x \in \R$ be the largest solution of $x=u+v(\log{x})^h$. Then
$$
x<\max\{2^h(u^{1/h}+v^{1/h}\log(h^hv))^h, 2^h(u^{1/h}+2e^2)^h\}.
$$
\end{lemma}

Let us assume from now on that $(n,m,z_1,\dots,z_s)$ is a solution to Diophantine equation \eqref{eq:main} satisfying the
assumptions made in Theorem \ref{thm1}. The next lemma gives
lower and upper bounds for the quantity $|u_n+u_m|$, as well as, upper and lower bounds for the exponents
$z_i$ in terms of $n$.
These bounds will be utilized in the proof of Theorem \ref{thm1}.

\begin{lemma} \label{lem:bounds}
There exist constants $c_1,\dots, c_5$ such that the following holds:
\begin{enumerate}[(i)]
\item $|u_n+u_m|<c_1|\alpha|^n.$
\item If $n>c_2$, then we have $z_i<\frac{2\la}{\log{p_i}} n$ for $i=1,\dots, s$. In particular, for all
$1\leq i \leq s$ we have $z_i<\frac{2\la}{\log{2}}n$.
\item Provided that $n>c_3$, we have
$$c_4|\alpha|^{n}<|u_n+u_m|.$$
\item Provided that $n>c_3$, we have
$$n < \frac{\sum_{i=1}^s{z_i\log{p_i}}}{\log |\alpha|}+c_5.$$
\end{enumerate}
\end{lemma}

\begin{proof} Let us start with \textit{(i)}. We note that since $|\alpha|>|\beta|$ and $|\alpha-\beta|=\sd$,
we get by the triangle inequality
$$
|u_n|=\left| \frac{a\al^n-b\be^n}{\al-\be} \right|<\frac{|a|+|b|}{\sd}|\al|^n.
$$
Thus, the above inequality leads to
$$
|u_n+u_m| \le |u_n|+|u_m|<\frac{|a|+|b|}{\sd}|\al|^n+\frac{|a|+|b|}{\sd}|\al|^m,
$$
which by $n \geq m$ gives $|u_n+u_m|<c_1|\alpha|^n$, with $c_1=2\frac{|a|+|b|}{\sqrt{\Delta}}$.

Next, we prove \textit{(ii)}. By combining equation \eqref{eq:main} and the just proved part \textit{(i)} of the lemma
we may write for every $1 \leq i \leq s$
$$
|w|p_i^{z_i} \le |w|p_1^{z_1} \dots p_s^{z_s}=|u_n+u_m|<c_1|\al|^n.
$$
Thus, by taking logarithms the above inequality implies
$$
z_i\log{p_i} \le n\log|\al|\left(1+\frac{c_2}{n}\right),
$$
with $c_2=\frac{\ls{\frac{c_1}{|w|}}}{\la}$. Assuming that $n>c_2$ we get
\begin{equation} \label{eq:boundsii}
z_i<\frac{2n\log{|\al|}}{\log{p_i}}
\end{equation}
and the first inequality of \textit{(ii)} follows. The second inequality of \textit{(ii)} is a simple
consequence of \eqref{eq:boundsii} by noting that $p_i \geq 2$ for every $1 \le i \le s$.

For proving \textit{(iii)} we note that since $\alpha$ and $\beta$ are roots of the quadratic monic polynomial
$x^2-Ax-B$ with $A,B \in \Z$
and $\Delta=A^2+4B>0$ we have
$$|\alpha| \geq \left| \frac{-A \pm \sqrt{A^2 -4B}}2\right|\geq \frac{|A| + \sqrt{|A|^2 +4|B|}}2\geq \frac{1+\sqrt{5}}{2}=\varphi.$$
Since we assume that $n\geq m$ we find
$$
\left|1+\frac{1}{\alpha^{n-m}}\right| \geq 1-\frac{1}{\varphi},
$$
hence we obtain by formula \eqref{eq:binet}
\begin{equation} \label{boundsiii_2}
\begin{split}
|u_n+u_m| & =\frac{1}{\sd}\left|a\alpha^n\left(1+\frac{1}{\alpha^{n-m}}\right)-(b\be^n+b\be^m)\right| \ge \\
 & \geq \frac{1}{\sd}\left||a||\al|^n\left(1-\frac{1}{\varphi}\right)-|b|(|\be|^n+|\be|^m)\right|.
\end{split}
\end{equation}
Let us distinguish between the case that $|\beta|\leq 1$ and $|\beta|>1$.

In the case that $|\beta| \leq 1$ inequality \eqref{boundsiii_2} yields
\begin{equation*}
\begin{split}
|u_n+u_m| & \ge\frac{1}{\sd}\left||a||\al|^n\left(1-\frac{1}{\varphi}\right)-2|b|\right| = \\
 & =
\frac{1}{\sd}|a||\al|^n\left(1-\frac{1}{\varphi}\right)\left|1-\frac{2|b|}{|a||\al|^n\left(1-\frac{1}{\varphi}
\right)}\right|.\\
\end{split}
\end{equation*}
Let us assume that $n$ is large enough such that
$1-\frac{2|b|}{|a||\al|^n\left(1-\frac{1}{\varphi}\right)}>\frac{1}{2}$,
i.e. $n>c'=\frac{\ls{\frac{4|b|\varphi}{|a|(\varphi-1)}}}{\log{|\alpha|}}$. Thus, we obtain
$$
|u_n+u_m|>\frac{1}{2\sd}|a||\al|^n\left(1-\frac{1}{\varphi}\right)=c_4|\alpha|^n,
$$
with $c_4=\frac{|a|(\varphi-1)}{2\varphi\sd}$.

Now let us assume that $|\beta| > 1$, then by $n\geq m$ and inequality \eqref{boundsiii_2} we get
\begin{equation*}
\begin{split}
|u_n+u_m| & \ge\frac{1}{\sd}\left||a||\al|^n\left(1-\frac{1}{\varphi}\right)-2|b||\be|^n\right| = \\
 & =
\frac{1}{\sd}|a||\al|^n\left(1-\frac{1}{\varphi}\right)\left|1-\frac{2|b||\be|^n}{|a||\al|^n\left(1-\frac{1}{
\varphi}\right)}\right|.\\
\end{split}
\end{equation*}
If we assume that $n$ is large enough, say
$n>c''=\frac{\ls{\frac{4|b|\varphi}{|a|(\varphi-1)}}}{\log{\frac{|\alpha|}{|\be|}}}$, then we get that
$\left|1-\frac{2|b||\be|^n}{|a||\al|^n\left(1-\frac{1}{\varphi}\right)}\right|> \frac{1}{2}$. Thus also in
this case we have
$$
|u_n+u_m|>\frac{1}{2\sd}|a||\al|^n\left(1-\frac{1}{\varphi}\right)=c_4|\alpha|^n,
$$
which proves \textit{(iii)} in the case that $|\beta| > 1$. All together we have proved \textit{(iii)} assuming
that $n>c_3=\max\{c',c''\}$.

Finally we turn to the proof of \textit{(iv)}. Let us note that by assuming that $n$ is large enough, i.e.
$n>c_3$, we may combine the inequality from \textit{(iii)}
with equation~\eqref{eq:main}. Thus we get
$$
c_4|\al|^n<|u_n+u_m|=|w|p_1^{z_1}\dots p_s^{z_s},
$$
and by taking logarithms this yields
$$
n\log{|\al|}<\log\frac{|w|}{c_4}+\sum_{i=1}^s{z_i\log{p_i}}
$$
which proves \textit{(iv)}, with $c_5=\frac{\log{\frac{|w|}{c_4}}}{\la}$.
\end{proof}

Before stating the next result let us introduce some further notations. Let $L$ be a number field and $\eta\in
L$. We denote by $h(\eta)$ as usual
the {\it absolute logarithmic Weil height} of $\eta$, i.e.
$$
h(\eta)= \dfrac{1}{d} \left( \log \vert a_{0} \vert + \sum\limits_{i=1}^{d} \ls \left(|\eta^{(i)}|\right) \right),
$$
where $a_{0}$ is the leading coefficient of the minimal polynomial of $\eta$ over $\Z$ and the $\eta^{(i)}$-s
are the conjugates of $\eta$ in $\C$.
We will use the following well known properties of the absolute logarithmic height without special reference:
\begin{equation*}
\begin{split}
h(\eta \pm \ga) \leq & \ h(\eta)+h(\ga)+\log{2},  \\
h(\eta\ga^{\pm 1}) \leq & \ h(\eta)+h(\ga),  \\
h(\eta^\ell)=& \ |\ell|h(\eta), \qquad \textrm{for} \; \ell \in \Z.
\end{split}
\end{equation*}

In general it is a very hard problem to find lower bounds for the height of elements in a number field of
given degree. However, in the case of quadratic fields the
problem can be solved easily:

\begin{lemma}\label{lem:height}
For an algebraic number $\alpha$ of degree two we have $h(\alpha)\geq 0.24$ or $\alpha$ is a root of unity.
\end{lemma}

\begin{proof}
Assume that the conjugate of $\alpha$ is $\beta$ and that $|\alpha|\geq |\beta|$.
If $\alpha$ is not an algebraic integer, then the leading coefficient $a_0$ of its minimal polynomial
satisfies $a_0\geq 2$ and we obtain
$$h(\alpha)\geq \frac 12 \log|a_0|\geq \frac{\log 2}2 > 0.34.$$
On the other hand if $\alpha$ is an algebraic integer, then $|\alpha|\geq \varphi$ and
$$h(\alpha)\geq \frac 12 (\ls|\alpha|+\ls|\beta|)\geq \frac{\log \varphi}2>0.24.$$
\end{proof}

\section{The Application of $p$-adic methods}\label{Sec:p-adic}

Our next goal is to prove an upper bound for $n$ in terms of $n-m$ and $\log n$. Therefore we utilize linear
forms in $p$-adic logarithms. In particular, we use the results due to
Bugeaud and Laurent \cite{Bugeaud:1996}. This result has the drawback that we have to assume that the algebraic numbers 
involved in the $p$-adic logarithms are linearly independent in contrast to the results due to
Yu \cite{Yu:1999, Yu:2013} which can be applied without this assumption. Thus dealing with the linearly
dependent case causes some technical difficulties. Nevertheless we use the
result due to Bugeaud and Laurent \cite{Bugeaud:1996} since their result yields rather small upper bounds and
dealing with the technical difficulties will pay off when we come to
the LLL-reduction step.

For a prime number $p$ denote by $\Q_p$ the field of $p$-adic numbers with the standard $p$-adic valuation
$\ordp$.
Further, let $\al_1,\al_2$ be algebraic numbers over $\Q$ and we regard them as elements of the field
$L=\Q_p(\al_1,\al_2)$.
We equip the field $L$ with the ultrametric absolute value $|x|_p=p^{-\nu_p(x)}$,
where $\nu_p$ denotes the unique extension to $L$ of the standard $p$-adic valuation $\ordp$ over $\Q_p$
normalized by $\nu_p(p)=1$ (we set $\nu_p(0)=+\infty$). Note, that for every non-zero $\delta \in L$ we have
$$
\nu_p(\delta)=\frac{\ordp(\Norm_{L/\Q_p}(\delta))}{d_{L}},
$$
where $d_{L}=[L:\Q_p]$ is the degree of the field extension $L/\Q_p$ and $\Norm_{L/\Q_p}(\delta)$ is the
norm
of $\delta$ with respect to $\Q_p$. Denote by $e$ the ramification index of the local field extension $L/\Q_p$
and by $f$ the residual degree of this extension.
Put
$$
\D=\frac{[\Q(\al_1,\al_2):\Q]}{f}
$$
and let
$$
h'(\alpha_i) \ge \max\left\{h(\al_i),\frac{\log{p}}{\D} \right\}, \ (i=1,2).
$$
Before stating the above mentioned result of Bugeaud and Laurent we introduce $p$-adic logarithms (see e.g.
\cite[Chapter 5.1]{Washington:CF} for a short but thorough treatment of this topic).
Let $\mathcal K$ be a field complete with respect to $|\cdot|_p$ and containing $\Q_p$. There exists a
function $\log_p(x)$ defined on all
of $\mathcal K$ such that $\log_p(xy)=\log_p(x)+\log_p(y)$. Moreover, for every $\xi\in\mathcal K$ with
$|\xi-1|_p<p^{-1/(p-1)}$
we can compute the $p$-adic logarithm by
$$\log_p(\xi)=-\sum_{i=1}^\infty \frac{(1-\xi)^i}i.$$
Moreover we have
\begin{equation}\label{eq:p-adic-log} |\log_p(\xi)|_p=|\xi -1|_p \end{equation}
or equivalently
\begin{equation} \label{eq:p-adic-log-2} \nu_p(\log_p(\xi)) = \nu_p(\xi-1) \end{equation}
provided that $|\xi-1|_p<p^{-1/(p-1)}$. Note that in the case that $|\xi-1|_p=p^{-1/(p-1)}$ we can replace in
equation \eqref{eq:p-adic-log} the $=$ by a $\leq$ sign,
(or equivalently in equation \eqref{eq:p-adic-log-2} the $=$ by a $\geq$ sign).

With these notations at hand, let us state a result due to Bugeaud and Laurent \cite[Corollary
1]{Bugeaud:1996}):

\begin{theorem} \label{Th:Bugeaud}
Let $b_1,b_2$ be positive integers and suppose that $\al_1$ and $\al_2$ are multiplicatively independent
algebraic numbers such that $\nu_p(\al_1)=\nu_p(\al_2)=0$. Put
\begin{equation}
\label{eq:bprime}
b':=\frac{b_1}{\D h'(\alpha_2)}+\frac{b_2}{\D h'(\alpha_1)}.
\end{equation}
Then we have
\begin{equation}
\label{eq:lowpadicbound}
\nu_p(\alpha_1^{b_1}\alpha_2^{b_2}-1)\leq \frac{24 p (p^f-1) \D^4}{(p-1)(\log
p)^4}B^2h'(\alpha_1)h'(\alpha_2),
\end{equation}
with
\begin{equation}
\label{eq:B}
B:=\max\left\{\log b'+\log\log p+0.4,10,\frac{10\log p}{\D}\right\}.
\end{equation}
\end{theorem}

We only use this result in the case where $L=\Q_p(\sqrt{\Delta})$, i.e. $1 \le d_L,\D,e,f \leq 2$ and obtain
as a simple Corollary

\begin{corollary}\label{cor:Bugeaud}
 Let $B'=\max\{b_1,b_2,p^{10},e^{10}\}$ and with the notations and assumptions of Theorem \ref{Th:Bugeaud} in
force we have
 $$
 \nu_p(\alpha_1^{b_1}\alpha_2^{b_2}-1) < C_1(p) h'(\alpha_1)h'(\alpha_2) (\log B')^2,
 $$
where
$$C_1(p)=\frac{947 p^f}{(\log p)^4}.$$
\end{corollary}

\begin{proof}
 First let us note that since $f,\D \leq 2$ one easily deduces that
 $$\frac{24 p (p^f-1) \D^4}{(p-1)(\log p)^4} < \frac{768 p^{f}}{\log^4 p}.$$
 Moreover,
 $$b'= \frac{b_1}{\D h'(\alpha_2)}+\frac{b_2}{\D h'(\alpha_1)} \leq \frac{b_1+b_2}{\log p},$$
 and since we assume that $B'\geq \max\{p^{10},e^{10}\}$ we get
 $$B=\max\left\{\log b'+\log\log p+0.4,10,\frac{10\log p}{\D}\right\} \leq \log (2 e^{0.4} B')\leq 1.11 \log
B'.$$
 Putting everything together yields the Corollary.
\end{proof}

Note that it is rather easy to compute $e$ and $f$. Indeed, let $\Delta_0$ be the square-free part of
$\Delta$, then we have $e=2$ if and only if $p|\Delta_0$ and
$f=2$ if and only if $\left(\frac{\Delta_0}p\right)=-1$, where $\left(\frac{\Delta_0}p\right)$ denotes the Legendre symbol.

\begin{proposition}\label{prop:p-adic}
 Let us assume that $n>m>3$, $P=\max_{1\leq i \leq s} \{p_i\}$ and that
 $$n>c_6:=\max\left\{c_3,17.5 \log|\alpha|(\max\{\log|2a\alpha|,\log|2b\beta|\}+0.24),P^{10},e^{10}\right\}.$$
 Then there exist constants $c_7$ and $c_{8,i}$ with $i=1,\dots,s$ such that
 $$n<c_7(n-m)(\log n)^2\qquad \text{and} \qquad  z_i<c_{8,i}(n-m)(\log n)^2,$$
 with $1 \leq i \leq s$.
\end{proposition}

\begin{proof}
 First, let us note that the first inequality is a direct consequence of the second inequality and Lemma
\ref{lem:bounds} \textit{(iv)} and $c_7$ can be easily computed
 once we have found the constants $c_{8,i}$ (see Section \ref{Sec:Appendix}). Therefore we are left to prove
the second inequality for each index $1\leq i \leq s$.
 Let us fix the index $i$. In order to avoid an overload of indices we drop the index $i$ for the rest of the
proof, i.e. we write $p=p_i$, $c_{8}=c_{8,i}$ and so on.
 Further, we will work in the field $L=\Q_p(\alpha,\beta)=\Q_p(\sqrt{\Delta})$.
 Since we assume $p\nmid \gcd(A,B)$ we have that $\nu_p(\alpha)=0$ or $\nu_p(\beta)=0$. Without loss of
generality we may
 assume that $\nu_p(\alpha)=0$, since we make no use of the fact that $|\alpha|>|\beta|$ in the whole proof of
the proposition.

 Let us consider equation \eqref{eq:main} and rewrite it as follows
 \begin{equation}\label{eq:p-adic-form}
\frac{b(\be^{n-m}+1)}{a(\al^{n-m}+1)}\left(\frac{\be}{\al}\right)^m-1=-w(\al-\be)\al^{-m}a^{-1}(\al^{n-m}+1)^{
-1}p_1^{z_1}\dots p_s^{z_s}
 \end{equation}
Let us denote by $\Lambda$ the left hand side of \eqref{eq:p-adic-form} and let us compute $\nu_p(\Lambda)$ by
considering the right hand side of equation \eqref{eq:p-adic-form}.
Since by assumption $\nu_p(\alpha)=\nu_p(w)=0$ we obtain
$$
\nu_p(\Lambda)=z-\nu_p(a(\al^{n-m}+1))+\nu_p(\alpha-\beta).
$$
Let us note that $\alpha-\beta$ is an algebraic integer, hence $\nu_p(\alpha-\beta)\geq 0$. Moreover, we have
\begin{align*}
\nu_p(a(\al^{n-m}+1))\leq & \frac{\log |a|+\log 2+(n-m)\log|\alpha|}{\log p} \\
\leq & (n-m)\frac{\log |2a\alpha|}{\log p}:=c_9(n-m).
\end{align*}
and obtain
\begin{equation}\label{eq:p-adic-lower}
 \nu_p(\Lambda)\geq z-c_9(n-m).
\end{equation}

Next, we will apply Corollary \ref{cor:Bugeaud} in order to bound $\nu_p(\Lambda)$ from above. Therefore let
us estimate the height of $\frac{b(\be^{n-m}+1)}{a(\al^{n-m}+1)}$.
Note that $\alpha$ and $\beta$ (resp. $a$ and $b$) are either rational integers or conjugate algebraic
integers in a real quadratic field. Therefore we conclude that
$h\left(\frac{\alpha}{\beta}\right)\leq \max\{\log |\alpha|,\log|\beta|\}$ provided that $\alpha$ and $\beta$
are either rational integers or conjugate algebraic
integers in a real quadratic field. Indeed this is clear in the rational case.
In the real quadratic case under the assumption that $|\alpha|\geq |\beta|$ we get
\begin{align*}
h\left(\frac{\alpha}{\beta}\right)= & \ \frac 12\left(\log |a_0|+\ls\left| \frac{\alpha}{\beta}\right|+
\ls\left| \frac{\beta}{\alpha}\right|\right)\\
\leq & \ \frac 12 \left(\log|\alpha|+\log|\beta|+\log|\alpha|-\log|\beta|\right)\\
= & \ \log|\alpha|.
\end{align*}
Note that the leading coefficient $a_0$ of the minimal polynomial of $\alpha/\beta$ satisfises $a_0|\alpha \beta=B$.
Therefore we obtain
\begin{equation}\label{eq:bound-x1}
\begin{split}
h\left(\frac{b(\be^{n-m}+1)}{a(\al^{n-m}+1)}\right)& \leq
\max\{\log|a|+\log|\alpha^{n-m}+1|,\log|b|+\log|\beta^{n-m}+1|\}\\
 & \leq\max\{\log|a|+(n-m)\log|\alpha|+\log 2,\\
 & \qquad \qquad  \log|b|+(n-m)\log|\beta|+\log 2|\}\\
 & \leq (n-m)\max\{\log|2a\alpha|,\log|2b\beta|\},
\end{split}
\end{equation}
hence
$$
h'\left(\frac{b(\be^{n-m}+1)}{a(\al^{n-m}+1)}\right)\leq (n-m)\max\{\log|2a\alpha|,\log|2b\beta|,\log
p\}:=c_{10}(n-m).
$$
Furthermore note that $h'(\alpha/\beta)\leq \max\{\log|\alpha|,\log p\}$.

We have to distinguish now between two cases, namely whether $\frac{b(\be^{n-m}+1)}{a(\al^{n-m}+1)}$ and
$\alpha/\beta$ are multiplicatively independent or not. Let us deal with the dependent case first.
In this case there exist co-prime integers $r$ and $s$ such that
 $$\frac{a(\alpha^{n-m}+1)}{b(\beta^{n-m}+1)}=\left(\frac{\alpha}{\beta}\right)^{r/s}.$$
 Let us find upper bounds for $|r|$ and $|s|$.

 First, we note that $\delta:=(\alpha/\beta)^{1/s}$ has to be still of degree two and is not a root of unity.
 Indeed $(\alpha/\beta)^{r/s}\in K$ therefore $(\alpha/\beta)^{r}\in K^s$. Since by assumption $r$ and $s$ are coprime
 we obtain that $(\alpha/\beta)\in K^s$ hence $\delta:=(\alpha/\beta)^{1/s}\in K$.
Hence
 $$0.24<\frac{h(\alpha/\beta)}{|s|}\leq \frac 1{|s|} \log |\alpha|$$
 and $|s|<4.2 \log|\alpha|$ by Lemma \ref{lem:height}. On the other hand we deduce from the upper bound
\eqref{eq:bound-x1} that
 \begin{multline*}
  4.2 (n-m)\max\{\log|2a\alpha|,\log|2b\beta|\}\log|\alpha|\geq |s|
h\left(\frac{b(\be^{n-m}+1)}{a(\al^{n-m}+1)}\right)\\
  =|r|\cdot h\left(\frac\alpha\beta\right) \geq |r| 0.24
 \end{multline*}
i.e.
$$|r|<17.5 (n-m)\log|\alpha|\max\{\log|2a\alpha|,\log|2b\beta|\}.$$
In particular, we have
$$-ms-r \leq m|s|+|r| \leq 17.5 n\log|\alpha|\left(\max\{\log|2a\alpha|,\log|2b\beta|\}+0.24\right)$$

Therefore rewriting \eqref{eq:p-adic-form} we obtain
\begin{equation*}
\begin{split}
\nu_p\left(\delta^{-ms-r}-1\right)= & \nu_p\left(-w(\al-\be)\al^{-m}a^{-1}(\al^{n-m}+1)^{-1}p_1^{z_1}\dots
p_s^{z_s}\right)\\
>& z-c_9(n-m)
\end{split}
\end{equation*}
Now, on supposing that $z \ge \frac{3}{2}+c_9(n-m)$ we may use property \eqref{eq:p-adic-log-2} of the
$p$-adic logarithm to obtain
$$\nu_p(\log_p \delta)+\nu_p(ms+r)> z-c_9(n-m).$$
If we assume that $n>17.5 \log|\alpha|\left(\max\{\log|2a\alpha|,\log|2b\beta|\}+0.24\right)$ we get by a very crude
estimate
\begin{equation*}
 \begin{split}
z<&\nu_p(\log_p \delta)+2\frac{\log n}{\log p}+c_9(n-m)\\
<&\left(\nu_p(\log_p \delta)+\frac{2}{\log p}+c_9\right)(n-m)\log n\\
:=& c_{11}(n-m)\log n
\end{split}
\end{equation*}
Note that
$$\nu_p(\log_p(\delta))=\nu_p\left(\frac 1s \log_p\left(\frac{\alpha}{\beta}\right)\right)\leq
\nu_p\left(\log_p\left(\frac{\alpha}{\beta}\right)\right).$$

Therefore we may assume that $\frac{b(\be^{n-m}+1)}{a(\al^{n-m}+1)}$ and $\alpha/\beta$ are multiplicatively
independent.
However, before applying Corollary \ref{cor:Bugeaud}, we have to ensure that
$$\nu_p\left(\frac{b(\be^{n-m}+1)}{a(\al^{n-m}+1)} \right)=\nu_p\left(\frac{\beta}{\alpha} \right)=0.$$
Let us assume for the moment that one of the $p$-adic valuations is not zero. But this would imply that the
$p$-adic valuation of the left hand side of \eqref{eq:p-adic-form} is zero, hence $z=0$.
Therefore we may apply Corollary \ref{cor:Bugeaud} and we obtain
\begin{equation}\label{eq:p-adic-upper}
 \nu_p(\Lambda)\leq C_1(p)\max\{\log|\alpha|,\log p\}c_{10}(n-m)\max\{\log n,10\log p,10\}^2.
\end{equation}
Comparing upper and lower bounds of $\nu_p(\Lambda)$ i.e. inequalities \eqref{eq:p-adic-lower} and
\eqref{eq:p-adic-upper} yields
$$C_1(p)\max\{\log|\alpha|,\log p\}c_{10}(n-m)(\log n)^2>z-c_9(n-m) $$
or by solving for $z$ followed by a crude estimate we obtain our upper bound for $z$:
$$z<\left(C_1(p)\max\{2\log|\alpha|,\log p\}c_{10}+c_9\right)(n-m)(\log n)^2:=c_8 (n-m)(\log n)^2.$$
\end{proof}

\section{The case $n=m$}\label{Sec:n=m}

Before we continue with the main line of the proof of Theorem \ref{thm1} let us consider the special case
$n=m$, i.e. the Diophantine equation
$$ 2u_n=w p_1^{z_1} \dots p_s^{z_s}.$$
Of course this equation has only a solution if $w$ is even or one of the primes, say $p_1=2$. Thus we are
reduced to consider Diophantine
equations of the type
\begin{equation}\label{eq:n=m_case}
 u_n=w'p_1^{z_1}\dots p_s^{z_s}
\end{equation}

Let us note that this type of equation has been considered by Peth\H{o} and de Weger~\cite{Pethoe:1986}.
They gave a practical method to solve this equation completely for a given binary, non-degenerate
sequence $u_n$ with discriminant $\Delta>0$, non-zero integer $w$ and primes $p_1,\dots,p_s$ by using $p$-adic
techniques.
In particular, they use linear forms in $p$-adic logarithms and a $p$-adic version of the Baker-Davenport
method \cite{Baker:1969}.
We also want to refer to the paper of Mignotte and Tzanakis~\cite{Mignotte:1991}, where also the case of
non-degenerate recurrences of arbitrary order is discussed.

\begin{proposition}[Peth\H{o} and de Weger \cite{Pethoe:1986}]\label{prop:n=m}
 Let $n,z_1,\dots,z_s$ be a solution to \eqref{eq:n=m_case}, then there exist explicitly computable constants
$c_{12,i}$ with $1\leq i \leq s$ and $c_{13}$
 such that $z_i<c_{12,i}$ for all $1\leq i \leq s$ and $n<c_{13}$.
\end{proposition}

The constant $c_{13}$ has been stated explicitly by Peth\H{o} and de Weger \cite[Theorem~4.1]{Pethoe:1986}.
However they used a version due to Schinzel \cite{Schinzel:1967} for lower bounds
in linear forms in $p$-adic logarithms that has been succeeded by newer developments e.g. the results due to
Yu \cite{Yu:1999,Yu:2013} and in particular by Bugeaud and Laurent \cite{Bugeaud:1996}.
For the sake of completeness and with respect to the newer developments we give a proof of Proposition
\ref{prop:n=m}.

\begin{proof}
 Since the proof due to Peth\H{o} and de Weger \cite{Pethoe:1986} and the similarity to the proof of
Proposition \ref{prop:p-adic} we only sketch the proof. As in the proof of Proposition \ref{prop:p-adic}
 we fix an index $i$ and drop it for the rest of the proof. We may assume without loss of generality that
$\nu_p(\alpha)=\nu_p(b/a)=\nu_p(\beta/\alpha)=0$.
 We rewrite equation \eqref{eq:n=m_case} and obtain
 \begin{equation}\label{eq:n=m-p-adic-form}
 \frac{b}{a} \cdot \left(\frac{\beta}{\alpha} \right)^n -1=w'(\alpha-\beta)a^{-1} \alpha^{-n} p_1^{z_1}\cdots
p_s^{z_s}.
 \end{equation}

 Let us assume for the moment that  $\frac{b}{a}$ and $\frac{\beta}{\alpha}$ are multiplicatively independent
and that $n>\max\{p^{10},e^{10}\}$.
 Let us put $\Lambda=\frac{b}{a} \cdot \left(\frac{\beta}{\alpha} \right)^n -1$. Estimating the $p$-adic
valuation on the right hand side of \eqref{eq:n=m-p-adic-form} yields
 $$\nu_p(\Lambda)\geq z-\frac{\ls |a|}{\log p}$$
 and an application of Corollary \ref{cor:Bugeaud} yields
 $$\nu_p(\Lambda)\leq C_1(p)h'\left(\frac ba \right)h'\left(\frac{\beta}{\alpha}\right)\max\{\log n,10\log
p,10\}^2,$$
 with
 $$h'\left( \frac ba \right)\leq \max\{\log |a|,\log |b|,\log{p}\}$$
 and
 $$h'\left( \frac {\beta}{\alpha} \right)\leq \max\{\log |\alpha|,\log |\beta|,\log{p}\}.$$
 Thus we get
 \begin{equation*}
 \begin{split}
 z<& \left(C_1(p)h'\left(\frac ba\right)h'\left(\frac{\beta}{\alpha}\right)+\frac{\ls|a|}{\log
p}\right)\max\{\log n,10\log p,10\}^2\\
 =&c_{14} \max\{\log n,10\log p,10\}^2.
 \end{split}
 \end{equation*}
 If we assume that $n>c_3$ we may apply Lemma \ref{lem:bounds} \textit{(iv)} and obtain either an absolute
bound for $n$ or an inequality
 of the form $n<c_{15} (\log n)^2$ and Lemma \ref{lem:pdw} yields an absolute upper bound for $n$. In any case
we obtain the Proposition.

 We are left with the case that  $\frac{a}{b}$ and $\frac{\alpha}{\beta}$ are multiplicatively dependent, i.e.
there are co-prime integers $s$ and $r$ such that
 $\left(\frac ab \right)^s=\left(\frac \alpha\beta \right)^r$ and let $\delta=\left(\frac \alpha\beta
\right)^{1/s}$. As in the proof of
 Proposition \ref{prop:p-adic} we deduce $|s|<4.2 \log|\alpha|$ and by a similar computation we obtain
$|r|<17.5 \log|\alpha|\max\{\log |a|,\log |b|, 1\}$.
 Hence we obtain from \eqref{eq:n=m-p-adic-form} the inequality
 $$\nu_p(\delta^{-ns-r}-1)\geq z-\frac{\ls |a|}{\log p}.$$
 Now on supposing $z \ge \frac{3}{2}+\frac{\ls |a|}{\log p}$ we may use the properties of the $p$-adic
logarithms (e.g. see formula \eqref{eq:p-adic-log-2}) and we obtain
 $$z<\nu_p(\log_p\delta)+\frac{\log(4.2 n \log |\alpha|+17.5\log|\alpha|\max\{\log |a|,\log
|b|, 1\})+\ls|a|}{\log p}.$$
 Assuming that
 \begin{equation*}
 n>17.5\log|\alpha|\max\{\log |a|,\log |b|, 1\}
 \end{equation*}
 we obtain $z<\left(\nu_p(\log_p\delta)+\frac{2}{\log{p}}+\frac{\ls|a|}{\log{p}} \right) \log n$, and by Lemma \ref{lem:bounds} \textit{(iv)} and the assumption
$n>c_3$ we obtain
 $$n<\frac{\sum_{i=1}^s {\left(\nu_{p_i}(\log_{p_i}\delta)\log{p_i}+(2+\ls|a|) \right) }   }{\log |\alpha|} \log n+c_5=c_{16}\log n+c_5$$
 and once again applying Lemma \ref{lem:pdw} yields an upper bound $c_{13}$ for $n$ (for an explicit
determination see Section \ref{Sec:Appendix}).
\end{proof}

 \section{The case $n>m$}\label{Sec:n>m}

 In order to solve the case $n>m$, we have to use lower bounds for linear forms in complex logarithms. The
currently best result suitable for our purpose is
 the following theorem due to Matveev \cite{Matveev:2000}.

 \begin{theorem}[Matveev \cite{Matveev:2000}] \label{Th:Matveev}
 Denote by $\eta_1,\ldots,\eta_n$ algebraic numbers, not $0$ or $1$, by $\log\eta_1,\ldots,\log\eta_n$
determinations of their logarithms,
 by $D$ the degree over $\Q$ of the number field $L = \Q(\eta_1,\ldots,\eta_n)$, and by $b_1,\ldots,b_n$
rational integers.
 Define $B'=\max\{|b_1|,\ldots,|b_n|\}$, and $A_i= \max\{D h(\eta_i),|\log\eta_i|, 0.16\}$ ($1\le i\le n$),
where $h(\eta)$ denotes
 the absolute logarithmic Weil height of $\eta$. Assume that the number
$\Lambda=b_1\log\eta_1+\cdots+b_n\log\eta_n$ does not vanish. Then
\[\log |\Lambda|\geq -C(n,\varkappa) D^2 A_1\cdots A_n \log (eD) \log (eB')\},\]
where $\varkappa=1$ if $\mathbb K \subset \R$ and $\varkappa =2$ otherwise and
\[C(n,\varkappa)=\min \left\{ \frac 1{\varkappa} \left( \frac 12 en \right) ^{\varkappa} 30^{n+3} n^{3.5},
2^{6n +20} \right\}. \]
\end{theorem}

Since the field $K=\Q(\sqrt{\Delta})$ of our interest is real and of degree at most $2$ and all the $\alpha_i$
will be positive we obtain
\begin{equation}\label{eq:Matveev_bound}
 \log |\Lambda|\geq -C_2(n) h(\eta_1)\cdots h(\eta_n) \log B',
\end{equation}
where
$$
C_2(n)= 2.31 \cdot 60^{n+3} n^{4.5}
$$
provided that $B'\geq 3$. Note that since Lemma \ref{lem:height} we can choose $A_i=2h(\eta_i)$.

The main purpose of this section is to find absolute upper bounds for $n$ and the exponents $z_i$ with $1\leq
i \leq s$.
For technical reasons we will assume that $n-m$ is not to small. In particular, we may assume that
$$n-m>\frac{\log\left(2\left(1+\frac{2|b|}{|a|}\right)\right)}{\log\left(\min\left\{\frac{|\alpha|}{|\beta|},
|\alpha|\right\}\right)}=c_{17}.$$
Indeed, due to Proposition \ref{prop:p-adic} the case that $n-m\leq c_{17}$  immediately implies
$n<c_7(n-m)(\log n)^2<c_7c_{17} (\log n)^2$. Thus by Lemma \ref{lem:pdw} we get that
$n\leq 4 c_7c_{17} \log(4c_7c_{17})^2$.

Hence we assume for the rest of this section that $n-m > c_{17}$.
Let us rewrite equation \eqref{eq:main} as
\begin{equation} \label{eq:linform}
|p_1^{z_1} \cdots p_s^{z_s}w(\al-\be){a}^{-1}\al^{-n}-1|=\frac{|-b\be^n+a\al^m-b\be^m|}{|a\al^n|}
\end{equation}
Let us estimate the right hand side of \eqref{eq:linform}. We obtain
\begin{align*}
 \frac{|-b\be^n+a\al^m-b\be^m|}{|a\al^n|}\leq & \left| \frac{2b \max\{\beta^n,\beta^m\}}{a\alpha^n}
\right|+|\alpha|^{m-n}\\
 &\leq \frac{2|b|}{|a|}\max\left\{\frac{|\beta|^n}{|\alpha|^n},\frac{|\beta|^m}{|\alpha|^n}\right\}
+|\alpha|^{m-n}\\
 &\leq \frac{2|b|}{|a|}\max\left\{\left(\frac{|\alpha|}{|\beta|}\right)^{m-n},|\alpha|^{m-n}\right\}
+|\alpha|^{m-n}\\
 &\leq \left(1+\frac{2|b|}{|a|}\right)\max\left\{\frac{|\be|}{|\al|},\frac{1}{|\al|} \right\}^{n-m}.
\end{align*}
Note that our assumption $n-m\geq c_{17}$ was chosen such that the right hand side of equation \eqref{eq:linform} is
smaller than $1/2$.
Since $|\log|x+1||\leq 2|x|$ provided $0\leq |x|\leq 1/2$ we obtain by taking logarithms on both sides of equation \eqref{eq:linform}
the inequality
\begin{equation}\label{eq:linform-ieq}
|\Lambda|:=|z_1\log p_1+\dots+z_s\log p_s+\log
|\gamma|-n\log|\alpha||<\frac{2\left(1+\frac{2|b|}{|a|}\right)}{\min\left\{\frac{|\alpha|}{|\beta|},
|\alpha|\right\}^{n-m}},
\end{equation}
where $\gamma=\frac{w\sd}{a}$. In order to apply Matveev's Theorem we have to ensure that $\Lambda$ does not
vanish. However this is established by the following Lemma:

\begin{lemma}\label{lem:vanishing}
 If $\Lambda=0$ then there exists a constant $c_{18}$ such that $n<c_{18}$.
\end{lemma}

In order to complete the proof of Theorem \ref{thm1} we postpone the proof of Lemma~\ref{lem:vanishing} to the
next section.
Now by Lemma \ref{lem:vanishing} we may assume that $\Lambda$ does not vanish and we may apply Matveev's
Theorem.
Before we apply Matveev's Theorem let us note that
$$B'=\max\{z_1,\dots,z_s,1,n\}\leq n \frac{2\log |\alpha|}{\log 2}< n^2,$$
provided that $n>\max\{c_2,\frac{2\log |\alpha|}{\log 2}\}$. Thus by comparing the upper bound from
inequality~\eqref{eq:linform-ieq}
with the lower bound of $|\Lambda|$ from Matveev's bound \eqref{eq:Matveev_bound} we obtain the inequality
\begin{multline*}
C_2(s+2)\log p_1 \dots \log p_s \log|\alpha| h(\gamma) 2 \log n>\\
(n-m)\log\left(\max\left\{\frac{|\alpha|}{|\beta|},|\alpha|
\right\}\right)-\log\left(1+\frac{2|b|}{|a|}\right),
\end{multline*}
i.e. there exists a constant $c_{19}$ (see Section \ref{Sec:Appendix} for an explicit form) such that
$n-m<c_{19}\log n$. In combination with
Proposition \ref{prop:p-adic} we obtain $n< c_7 c_{19} (\log n)^3$, provided that $n>c_6$.
Once again using Lemma \ref{lem:pdw} yields an upper bound $c_{20}$ for $n$ and Lemma \ref{lem:bounds} yields
an upper bound $z_i<c_{21,i}$
for each $1\leq i \leq s$.

 With the notations from above in force we have proved so far the following proposition, which clearly implies
Theorem \ref{thm1}:

 \begin{proposition}\label{prop:explicit}
 Under the assumptions of Theorem \ref{thm1} we have
 $$\max\{n,m\}< c_{20}\;\; \text{and}\;\; z_i<c_{21,i}$$
 for $i=1,\dots, s$.
 \end{proposition}

 \section{Proof of Lemma \ref{lem:vanishing}}

First, let us note that $\Lambda=0$ implies that
\begin{equation} \label{eq:Lambda-zero}
b\beta^n-a\alpha^m+b\beta^m=0.
\end{equation}
In the case that the $\Delta>0$ is not a perfect square $\alpha$ and $\beta$ respectively $a$ and $b$ are
algebraic conjugate and we obtain by conjugating the left side of equation \eqref{eq:Lambda-zero}
$$
-a\al^n+b\be^m-a\al^m=0
$$
which yields together with the original equation \eqref{eq:Lambda-zero} the relation $-b\be^n=a\al^n$. Therefore $n<\frac{\log |b/a|}{\log
|\alpha/\beta|}=c_{22}$.

Therefore we may assume that $\Delta$ is a perfect square and that $\alpha,\beta,a,b$ are all rational integers. Assume
for the moment that $|\beta|=1$.
Then \eqref{eq:Lambda-zero} turns into $a \alpha^m =0$ or $a \alpha^m=2b$. The first case would imply that $u_n$ is degenerate
and the second case has been excluded.
Thus we may assume that $|\beta|\geq 2$.

Under these assumptions we obtain from \eqref{eq:Lambda-zero} by dividing through $b\beta^n$ the inequality
$$\left|\frac{|a|}{|b|}\frac{|\alpha|^m}{|\beta|^{n}}-1\right|\leq |\beta|^{m-n}$$
and by taking logarithms we get
\begin{equation} \label{eq:linform2}
\Lambda:=\log\left|\frac ab\right|+m\log |\alpha|-n\log|\beta|<\frac{2}{|\beta|^{n-m}}
\end{equation}
Let us assume for the moment that $\Lambda=0$. Since $a,b,\al,\be$ are integers $\Lambda=0$ implies that
$a\al^m=\pm b\be^n$ and in combination with equation \eqref{eq:Lambda-zero} this leads either to $b\be^m=0$ or $\be^{n-m}=-\frac{1}{2}$.
But, both cases are contradictions in view of $|\beta| \geq 2$ and $n>m$.

Hence we may suppose that $\Lambda \neq 0$ and therefore we may apply Matveev's Theorem \ref{Th:Matveev} to
the left side of inequality  \eqref{eq:linform2} and obtain
$$\log|\Lambda|>-C_2(3)\max\{0.16,\ls\max\{|a|,|b|\}\}\log|\alpha|\log|\beta|\log{n}.$$
On the other hand we have
$$\log|\Lambda|<\log 2 -(n-m)\log|\beta|.$$
Hence we get that $n-m<c_{23}\log{n}$ (see Section \ref{Sec:Appendix} for an explicit determination of $c_{23}$).
Combining this with the results of Proposition \ref{prop:p-adic}
we get $n<c_7c_{23}(\log n)^3$ provided that $n>c_6$. Thus we find an absolute upper bound $c_{18}$ for $n$ by applying
Lemma \ref{lem:pdw}. Let us note that in Section \ref{Sec:Algorithm}
we explain in more detail how to handle this case in practice.

\section{Reduction of our bounds}\label{Sec:Reduction}

This section is devoted to the problem of reducing the rather large bounds obtained by Theorem \ref{thm1} and
Proposition \ref{prop:explicit} respectively. In this paper we will make use
of the LLL-algorithm due to Lenstra, Lenstra and Lov\'asz \cite{Lenstra:1982} to reduce our upper bounds for
$n-m$ and $z_1,\dots,z_s$.
In the $p$-adic case we use instead of approximation lattices an idea due to Peth\H{o} and de Weger
\cite[Algorithm A]{Pethoe:1986}.

\subsection{Real approximation lattices}\label{Sec:Red1}
Let us start with gathering some basic facts on LLL-reduced bases and approximation lattices. Therefore let
$\LL\subseteq \R^k$ be a $k$-dimensional lattice with LLL-reduced basis
$b_1,\dots,b_k$ and let $B$ be the matrix with columns $b_1,\dots, b_k$. Moreover, we denote by
$b^*_1,\dots,b^*_k$ the orthogonal basis of $\R^k$ which we obtain
by applying the Gram-Schmidt process to the basis $b_1,\dots,b_k$. In particular, we have that
$$b^*_i=b_i-\sum_{j=1}^{i-1}\mu_{i,j}b^*_j, \qquad \mu_{i,j}=\frac{\langle b_i,b_j\rangle}{\langle
b_j^*,b_j^*\rangle}.$$
Further, let us define
\begin{equation*}
l(\LL,y)=\begin{cases}
\min_{x \in \LL} \{\|x-y\|\}, & y \not\in \LL \\
\min_{0 \neq x \in \LL} \{\|x\|\}, & y \in \LL,
\end{cases}
\end{equation*}
where $\|\cdot\|$ denotes the euclidian norm on $\R^k$. It is well known,
that by applying the LLL algorithm it is possible to give in a polynomial time
a lower bound for $l(\LL,y) \geq \tilde c_1$ (see e.g. \cite[Section 5.4]{Smart:DiGL}).

\begin{lemma}\label{lem:lattice}
 Let $y\in \R^k$, $z=B^{-1}y$ and if $y\not\in\LL$ let $i_0$ be the largest index such that $z_{i_0}\neq 0$.
Put $\sigma=\{z_{i_0}\}$, where
 $\{\cdot\}$ denotes the distance to the nearest integer, and in case that $y\in \LL$ we put $\sigma=1$.
Moreover, let
 $$\tilde c_2=\max_{1\leq j\leq k}\left\{\frac{\|b_1\|^2}{\|b_j^*\|^2} \right\}.$$
 Then we have
 $$l(\LL,y)^2\geq \tilde c_2^{-1}\sigma \|b_1\|^2=\tilde c_1.$$
\end{lemma}

In our applications suppose we are given $\eta_0,\eta_1,\dots,\eta_k$ real numbers linearly independent over
$\Q$ and two positive constants
$\tilde c_3,\tilde c_4$ such that
\begin{equation} \label{eq:redform1}
|\eta_0+x_1\eta_1+\dots+x_k\eta_k| \le \tilde c_3\exp(-\tilde c_4H),
\end{equation}
where the integers $x_i$ with $1\leq i \leq k$ are bounded by $|x_i| \leq X_i$ with $X_i$ given upper bounds for all $1 \leq i \leq k$.
Set $X_0=\max_{1 \le i \le s}\{X_i\}$. The basic idea in such a situation, due to de Weger \cite{deWeger:1987},
is to approximate the linear form
\eqref{eq:redform1} by an approximation lattice. Namely, we consider the lattice $\LL$ generated by the
columns of the matrix
$$
\begin{pmatrix}
    1 & 0 & \dots & 0  & 0 \\
    0 & 1 & \dots & 0  & 0 \\
    \vdots & \vdots & \vdots & \vdots & \vdots \\
    0 & 0 & \dots & 1  & 0 \\
    \lfloor{C\eta_1}\rfloor & \lfloor{C\eta_2}\rfloor & \dots & \lfloor{C\eta_{k-1}}\rfloor &
\lfloor{C\eta_k}\rfloor
\end{pmatrix}
$$
where $C$ is a large constant usually of the size about $X_0^k$. Let us assume that we have an LLL-reduced
basis $b_1,\dots,b_k$ of $\LL$ and that we have a lower bound
$l(\LL,y)\geq \tilde c_1$ with $y=(0,0,\dots ,-\lfloor{C\eta_0}\rfloor)$. Then we have with these notations the
following Lemma concerning inequality \eqref{eq:redform1} (c.f. \cite[Lemma VI.1]{Smart:DiGL}):

\begin{lemma}\label{lem:real-reduce}
Assume that $S=\sum_{i=1}^{k-1}X_i^2$ and $T=\frac{1+\sum_{i=1}^k{X_i}}{2}$. If $\tilde c_1^2 \ge T^2+S$, then
we have either $x_1=x_2=\dots=x_{k-1}=0$ and $x_k=-\frac{\lfloor{C\eta_0}\rfloor)}{\lfloor{C\eta_k}\rfloor)}$
or
\begin{equation} \label{eq:reduction-real}
H \leq \frac{1}{\tilde c_4}\left(\log(C\tilde c_3)-\log\left(\sqrt{\tilde c_1^2-S}-T\right) \right).
\end{equation}
\end{lemma}

We will apply Lemma \ref{lem:real-reduce} to inequality  \eqref{eq:linform-ieq}, i.e. $k=s+1$ and $x_i=z_i$ for
$i=1,\dots,s$ and $x_{s+1}=n$. For the coefficients $\eta_i$ with $0\leq i \leq s+1$ we take
$\eta_0=\log \gamma$, $\eta_i=\log p_i$ for $1\leq i \leq s$ and $\eta_{s+1}=\log |\alpha|$. Furthermore we
have $\tilde c_3=2\left(1+\frac{2|b|}{|a|}\right)$
and $\tilde c_4=\log\left(\min\left\{\frac{|\alpha|}{|\beta|},|\alpha|\right\}\right)$.

For sure $\eta_1,\dots,\eta_s$ are linearly independent over $\Q$ and it is
easy to find a multiplicative dependence of $\gamma$ and/or $\alpha$ of the $p_i$ with $i=1,\dots,s$. However,
let us assume for the moment that $\log|\gamma|,\log p_1,\dots,\log p_s$
and $\log|\alpha|$ are linearly independent over $\Q$ (a similar argument holds if this is not the case). Let
us also assume that $C$ was chosen large enough such that
$$\tilde c_1 ^2 > \sum_{i=1}^s c_{21,i}^2+\left(\frac{\sum_{i=1}^s c_{21,i} +c_{20}}2\right)^2=T^2+S.$$
Then we get by Lemma \ref{lem:real-reduce} either a new bound for $H=n-m$ or $z_1=\dots z_s=0$ and our
Diophantine equation \eqref{eq:main} reduces to $u_n+u_m=w$ wich can be easily solved by
brute force if $w$ is not too large.

Let us note that a new bound for $n-m$ immediately yields new upper bounds for $z_i$ with $1\leq i \leq s$ and
$n$ by Proposition \ref{prop:p-adic}.
We can apply the same trick once again with these new bounds and obtain again a further reduction of the
bounds for $z_i$ with $1\leq i \leq s$ and $n$. We can repeat this as long as our reduction method
yields smaller upper bounds for $n-m$.

\subsection{The $p$-adic reduction method}\label{Sec:Red2}

In the $p$-adic case it is also possible to use so-called $p$-adic approximation lattices (see e.g. \cite[Section 5.6]{Smart:DiGL}).
However we are only interested in a very special case
namely in the situation appearing in inequality \eqref{eq:p-adic-form} under the assumption that
$n-m=t$ is a fixed (small) integer. The following is based on an idea due to Peth\H{o} and de Weger
\cite[Algorithm A]{Pethoe:1986}. We reproduce theire idea and fit it into our frame work.

Let us fix an index $i$ with $1\leq i \leq s$. In order to avoid an overloaded notation we drop the index $i$
for the rest of this section.
Furthermore let us assume that $N$ and $Z$ are given upper bounds for $n$ and $z$ respectively.
Let us consider the $p$-adic valuation of the left and right side of \eqref{eq:p-adic-form}. Then we obtain
(as in Section \ref{Sec:p-adic})
$$\nu_p\left(\tau(t) \left(\frac{\beta}{\alpha}\right)^m-1\right)=z-\nu_p(a(\al^t+1))+\nu_p(\alpha-\beta)=z -
z_0,$$
where
$$\tau(t)= \frac{b(\be^t+1)}{a(\al^t+1)}$$
and $z_0$ is easily computable for a fixed $t$. Let us assume that $z \geq 3/2+z_0=\tilde c_5$, then we may
take $p$-adic logarithms and obtain
\begin{equation}\label{eq:p-adic-reduction}
\nu_p\left(\log_p(\tau(t))-m\log_p(\alpha/\beta)\right)=z-z_0.
\end{equation}
We distinguish now between two cases whether $\log_p(\tau(t))=0$ or not.

Let us discuss the rather unlikely case that $\log_p(\tau(t))=0$ first. In this case equation~\eqref{eq:p-adic-reduction} turns into
$$\nu_p(m\log_p(\alpha/\beta))=\nu_p(m)+\nu_p(\log_p(\alpha/\beta))=z-z_0,$$
i.e.
$$z<\frac{\log m}{\log p}+\nu_p(\log_p(\alpha/\beta))+z_0<\frac{\log N}{\log
p}+\nu_p(\log_p(\alpha/\beta))+z_0=\tilde c_6$$
and $\max\{\tilde c_5,\tilde c_6\}$ is a new upper bound for $z$.

Let us turn to the case that $\log_p(\tau(t))\neq 0$. Since $\alpha$ and $\beta$ are
conjugate in $\Q(\sqrt{\Delta})\subset \Q_p(\sqrt{\Delta})$ also
$\log_p(\alpha)$ and $\log_p(\beta)$ are conjugate, hence
$\log_p(\alpha/\beta)=\log_p(\alpha)-\log_p(\beta)\in \sqrt{\Delta} \Q_p$.
Similarly we get $\log_p(\tau(t))\in \sqrt{\Delta} \Q_p$, since $\tau(t)$ is the quotient of conjugates. In
particular, we get
$$\zeta=\log_p(\tau(t))/\log_p(\alpha/\beta)=u_0+u_1p+u_2p^2+\dots \in\Q_p.$$
Let $r$ be the smallest possible exponent such that $p^r>N$ and let $0\leq m_0 < p^r$ be the unique integer
such that $m_0\equiv \zeta \mod p^r$.
Moreover, let $R$ be the smallest index $\geq r$ such that $u_R\neq 0$, if such an index exists. Then we get
\begin{align*}
z-z_0=&\nu_p\left(\log_p(\tau(t))-m\log_p(\alpha/\beta)\right)\leq
\nu_p\left(\log_p(\tau(t))-m_0\log_p(\alpha/\beta)\right)\\
=&\nu_p (\log_p(\tau(t))+\nu_p\left(1-(\zeta-u_Rp^R-\dots)\frac{\log_p(\alpha/\beta)}{\log_p(\tau(t))}
\right)\\
=&\nu_p (\log_p(\tau(t))+R+\nu_p\left(\frac{\log_p(\alpha/\beta)}{\log_p(\tau(t))} \right).\\
\end{align*}
Therefore we get a new upper bound for $z$ namely
$$z\leq \nu_p(\log_p(\tau(t))+R+\nu_p\left(\frac{\log_p(\alpha/\beta)}{\log_p(\tau(t))} \right).$$
Let us discuss the case, in which $R$ does not exist. In this case we would obtain that
$\zeta=m_0$ is an integer, hence
$$\log_p(\tau(t))=m_0\log_p\left(\alpha/\beta)\right),$$
and \eqref{eq:p-adic-reduction} turns into
$$
\nu_p\left(\log_p(\tau(t))-m\log_p(\alpha/\beta)\right)=\nu_p\left((m-m_0)\log_p(\alpha/\beta)\right) =z-z_0.
$$
Therefore we are in a similar situation as in the case that $\log_p(\tau(t))=0$ and we also get in this case a
new upper bound for $z$.

In any case we get for each index $i$ with $1\leq i \leq s$ a new upper bound $z_i$ and therefore by Lemma
\ref{lem:bounds} \textit{(iv)} a new upper bound for $n$.
We can repeat the procedure with these new, small upper bounds as long as we get smaller upper bounds for $n$.

\section{The Algorithm}\label{Sec:Algorithm}

At this point the inclined reader may have already a farely well idea how to solve Diophantine equations of
type \eqref{eq:main} in practice. However, let us
summerize the key steps and give some comments on the practical implementation. The algorithm to solve
Diophantine equation \eqref{eq:main} can be splitted up into
six key steps:\\

\noindent\textbf{Step I - First upper bounds for $n$ and $z_i$ for all $1\leq i\leq s$:} Compute one after the
other the constants $c_1,\cdots $ until you have found the
upper bounds $n<c_{20}$ and $z_i<c_{21,i}$ with $1\leq i \leq s$. Using the explicit determination of the
constants given in Section \ref{Sec:Appendix} this is straight forward.\\

\noindent\textbf{Step II - The case that $n=m$:} If $w$ and all $p_i$ with $1\leq i \leq s$ are odd, this case
has no solution. Otherwise in case that $w$ is even we replace $w$ by $w/2$ or
if $p_1=2$ we replace $z_1$ by $z_1-1$ and compute upper bounds for $n<c_{13}$ and $z_i<c_{12,i}$ for $1\leq i
\leq s$. With these upper bounds we use our $p$-adic reduction procedure from
Section \ref{Sec:Red2} with $t=1$ and $\tau(t)=b/a$ and obtain a small bound for $n$, which is in most cases
small enough to performe a brute force search.\\

\noindent\textbf{Step III - The case that $b\beta^n-a\alpha^m+b\beta^m=0$:} If $\Delta$ is not a perfect
square or $\beta=\pm 1$ there are no solutions in this case and we may omitt this step.
Otherwise we may assume that $\alpha,\beta,a$ and $b$ are rational integers and that $|\beta|\geq 2$.

In the case that $\beta\nmid \alpha$ we obtain the equation
$$\frac{a}{b} \left(\frac \alpha\beta\right)^m=\beta^{n-m}+1$$
We may assume that $\alpha/\beta=P/Q$ with $Q>1$ and $P,Q$ are coprime integers. Then $m\leq \log Q/\log a$
since otherwise the left hand side cannot be an integer. This presumably small bound for $m$ yields
also a small bound for $n$ and a brute force search will resolve this case.

Therefore we may assume that $\alpha/\beta=P$ is an integer and let us assume that $p$ is the largest prime
factor dividing $P$ and let us assume for the moment that $m>1$.
Then we have that
\begin{align*}
 m-\nu_p(b)=& \nu_p(\beta^{n-m}-1)\\
 =&\nu_p(\log_p(\beta))+\nu_p(n-m)\\
 <& \nu_p(\log_p(\beta))+\log (c_{23} \log c_{18}).
\end{align*}
Usually this will yield a reasonable small bound for $m$ and $n$ to perform a brute force search. If the new
upper bound $\max\{n,m\}<\tilde C$ is still large one may use once again
the above inequality to obtain
\begin{align*}
 m-\nu_p(b)=&\nu_p(\log_p(\beta))+\nu_p(n-m)\\
 <& \nu_p(\log_p(\beta))+\frac{\log(\tilde C)}{\log p}.
\end{align*}
This may be repeated until no further improvement on the bound for $\max\{n,m\}$ is possible and one has to
perform a brute force search for possible solutions to Diophantine equation \eqref{eq:main}.\\

\noindent\textbf{Step IV - First reduction of $n$:}
Perform the reduction step described in Section~\ref{Sec:Red1} to obtain upper bounds $n<N$, $z_i<Z_i$ for
$1\leq i \leq s$ and $n-m\leq T_0$. \\

\noindent\textbf{Step V - Second reduction of $n$:}
Perform the reduction step described in Section~\ref{Sec:Red2} for each $1\leq t \leq T_0$ and obtain for each
$t$ upper bounds $n<\tilde N_t$ and $z_i<\tilde Z_{i,t}$ for $1\leq i \leq s$.
Thus we get new upper bounds for $n<\tilde N=\max_{1\leq t\leq T_0}\{\tilde N_t\}$ and $z_i<\tilde
Z_i=\max_{1\leq t\leq T_0}\{\tilde Z_{i,t}\}$ for $1\leq i \leq s$.\\

\noindent\textbf{Step VI - Brute force search:}
The upper bounds $\tilde N$ and $\tilde Z_i$ for $1\leq i \leq s$ obtained in Step V are usually rather small.
However replacing $c_{23}$ by $\tilde N$ and $c_{22,i}$ by $\tilde Z_i$ for $1\leq i \leq s$
we can go back to Step IV and after performing the reduction steps IV and V again maybe we obtain sharper
bounds for $n$ and $z_i$ with $1\leq i \leq s$. This can be repeated until now further improvement
is possible. And we have to check the remaing cases by a brute force search. For instance we compute for all
$1\leq n\leq m \leq \tilde N$ the values of $u_n+u_m$ and write them into a list $\mathcal L$. For each element from
the list $\mathcal L$ we perform a trial division including the primes $p_1,\dots,p_s$. If $\tilde N$ and $P=\max\{p_1,\dots,p_s\}$ are
not unusually large, say $\tilde N,P< 10000$ this brute force search can be done within a reasonable time (see Section \ref{Sec:Example}).

\section{An example}\label{Sec:Example}

In this section we illustrate our algorithm by two examples. We completely solve Diophantine equation \eqref{eq:main}
in the case that $w=1$, $p_1,\dots,p_{46}$ are all primes smaller than $200$ and $u_n$ is the Fibonacci sequence
or the Lucas sequence, respectively. We have the following theorem:

\begin{theorem} \label{th:example}
\begin{enumerate}[(i)]
\item Let $\{F_n\}_{n \ge 0}$ be the Fibonacci sequence defined by $F_0=0, F_1=1$ and $F_n=F_{n-1}+F_{n-2}$ for
$n \ge 2$. Consider the equation
\begin{equation} \label{eq:fib}
F_n+F_m=2^{z_1}3^{z_2} \dots 199^{z_{46}}
\end{equation}
in non-negative integer unknowns $n,m,z_1,\dots,z_{46}$ with $n \ge m$. Then there are $325$ solutions
$(n, m, z_1 , \dots , z_{46})$ and there exists no solution with $\max\{n,m\}>59$.

\item Let $\{L_n\}_{n \ge 0}$ be the Lucas sequence defined by $L_0=2, L_1=1$ and $L_n=L_{n-1}+L_{n-2}$ for
$n \ge 2$. Consider the equation
\begin{equation} \label{eq:lucas}
L_n+L_m=2^{z_1}3^{z_2} \dots 199^{z_{46}}
\end{equation}
in non-negative integer unknowns $n,m,z_1,\dots,z_{46}$ with $n \ge m$. Then there are $284$ solutions
$(n, m, z_1 , \dots , z_{46})$ and there exists no solution with $\max\{n,m\}>63$.
\end{enumerate}
\end{theorem}

We refrain from giving a list of all $325$ and $284$ solutions to Diophatine equations \eqref{eq:fib} and \eqref{eq:lucas}, respectively.
Let us note that a brute force computer search for all solutions to \eqref{eq:fib} and \eqref{eq:lucas} with  $m\leq n\leq 59$ and $m\leq n\leq 63$
respectively is a matter of a few seconds on a usual PC.

However looking through the solutions to Diophatine equations \eqref{eq:fib} and \eqref{eq:lucas} respectively we noticed the following
interesting facts.

\begin{itemize}
 \item $z_1 \leq 8$ and $z_2 \leq 6$ in both cases.
 \item In the case of equation \eqref{eq:fib} we have that $z_3,z_5\leq 3$ and for all other indices $\neq 1,2,3,5$ we have that $z_i\leq 2$.
 \item In the case of equation \eqref{eq:lucas} we have that $z_3\leq 3$ and for all other indices $i\geq 4$ we have that $z_i\leq 2$.
 \item For all solutions to equation \eqref{eq:fib} the exponents of $79,83,131,139,163$ and $167$ are always zero.
 \item For all solutions to equation \eqref{eq:lucas} the exponent of $73,149,157,173,181,191$, $193$ and $197$ are always zero.
\end{itemize}

\begin{proof}[Proof of Theorem \ref{th:example}]
We will give the computational details of our algorithm only in the resolution of Diophantine equation \eqref{eq:fib}.
The resolution of Diophantine equation \eqref{eq:lucas} is similar and of the same computational time.
In the proof follow the steps given in the preceding section. An easy computation shows that if $u_n$ is the Fibonacci sequence, then we have
$a=b=1, \alpha=\frac{1+\sqrt{5}}{2}, \beta=\frac{1-\sqrt{5}}{2}, \sqrt{\Delta}=\alpha-\beta=\sqrt{5}$.\\

\noindent\textbf{Step I - First upper bounds for $n$ and $z_i$ for all $1\leq i\leq s$:}
By using the explicit determination of the constants given in Section \ref{Sec:Appendix} we
obtain by a straight forward computation that
\begin{equation} \label{eq: fib1}
\max\{n,z_i\}<2.6 \cdot 10^{117}, \quad (1 \le i \le 46).
\end{equation}\\

\noindent\textbf{Step II - The case that $n=m$:}
It is clear, that in this case we may suppose that $z_1 \ge 1$, since otherwise we do not have a solution.
Now, we replace $z_1$ with $z_1-1$ and compute the initial upper bounds $n<c_{13}$ and $z_i<c_{12,i} \ (1 \le
i \le 46)$ and get $\max\{n,z_i\}<N:=1.4 \cdot 10^{23}$. For every $p=p_i \ (1 \le i \le 46)$ we perform the $p$-adic
reduction procedure
described in Section \ref{Sec:Red2} with $t=1$ and $\tau(t)=b/a=1$. Since $\alpha-\beta=\sqrt{5}$ and $a=1$
we obtain that $z_0=\nu_p(a(\alpha^t+1))-\nu_p(\alpha-\beta) \le 1$ and $\log_{p}(\tau(t))=\log_{p}(1)=0$
holds for every $p=p_i \ (1 \le i \le 46)$.
Further, $\nu_p(\log_p(\alpha/\beta)) \le 2$ and hence we get that
\begin{equation*}
z=z_i<\frac{\log{(1.4 \cdot 10^{23})}}{\log{p}}+2+1=Z_{0,i}  \quad (1 \le i \le 46).
\end{equation*}
Note that the values $Z_{0,i}$ form a vector
\begin{multline*}
Z_{0}=(79,51,36,30,25,23,21,21,19,18,18,17,17,17,16,16,16,15,15,15,15,
15,\\15,14,14,14,14,14,14,14,14,13,13,13,13,13,13,13,13,13,13,13,13,13,13,13).
\end{multline*}
Since in our case $c_5<6$, we infer by \textit{(iv)} of Lemma \ref{lem:bounds} that
\begin{equation*}
n<\frac{\sum_{i=1}^{46}Z_{0,i}(\log{p_i})}{\log\left(\frac{1+\sqrt{5}}{2}\right)}+6 \le 6100.
\end{equation*}
By repeating the above reduction step once more, now with $N=6100$, we obtain $n \le 1771$.\\

\noindent\textbf{Step III - The case that $b\beta^n-a\alpha^m+b\beta^m=0$:}
One can easily see that this case cannot occur if $\{u_n\}_{n \ge 0}=\{F_n\}_{n \ge 0}$.\\

\noindent\textbf{Step IV - First reduction of $n$:} We may suppose that $n > m$.
We apply the reduction method described in Section \ref{Sec:Red1} with the initial upper bound
provided by Step I. Namely $\max\{n,z_i\}<X_{00}:=2.6 \cdot 10^{117}$ and choose $C:=10^{6300}$. Further,
put $\tilde{c}_3=6$, $\tilde{c}_4=\log{\frac{1+\sqrt{5}}{2}}$, $k=47$, $x_i=z_i$ with $1 \le i \le 46$, $x_{47}=n$
and $\eta_0=\log(\sqrt{5})$, $\eta_i=\log{p_i}$ with $1 \le i \le 46$. After performing the LLL algorithm we found
that $(\tilde{c}_2)^{-1}:=1/25000$ and $\sigma:=0.49$ is an appropriate choice.
Thus, we may apply Lemma \ref{lem:real-reduce} and we obtain a reduced bound $n-m\leq T_1=30000$.

Inserting this bound for $n-m$ into Proposition \ref{prop:p-adic} we are able to derive a new, better
bound for $\max\{n,z_i\}$.
An easy calculation shows that $\max\{n,z_i\}<X_{01}:=199^{10}$ is an appropriate choice.
By repeating once more the above procedure, with $X_{01}:=199^{10}, C:=10^{1280}, (\tilde{c}_2)^{-1}:=1/81550$
and $\sigma:=0.49$ we obtain $n-m \le T_2=6010$.\\

\noindent\textbf{Step V - Second reduction of $n$:}
We perform the reduction step described in Section~\ref{Sec:Red2} for each $1\leq t \leq T_2=6010$. Since
$\max\{n,z_i\}<X_{01}:=199^{10}$ we apply
the reduction step described in Section~\ref{Sec:Red2} for each prime $p=p_i$ with $1 \le i \le 46$ and for each
$1\leq t \leq T_2=6010$.
In particular, we obtain that $z=z_i \le Z_{1,i}$ with $1 \le i \le 46$, where the $Z_{1,i}$ form a vector
\begin{multline*}
 Z_1=(92, 57, 41, 37, 24, 27, 25, 19, 22, 20, 17, 21, 18, 20, 18, 20, 16, 16, 17, 14, 17, 14,\\
15, 15, 15, 13, 16, 16, 15, 17, 14, 13, 14, 14, 13, 12, 14, 14, 14, 14, 12, 13, 13, 14, 15, 14).
\end{multline*}
Since in our case $c_5<6$ we infer by \textit{(iv)} of Lemma \ref{lem:bounds} that
$$
n<\frac{\sum_{i=1}^{46}Z_{1,i}(\log{p_i})}{\log(\frac{1+\sqrt{5}}{2})}+6 \le 6600.
$$
By repeating the above reduction step once more, now with $X_{02}=6600$, we obtain $n \le 2300$.\\

\noindent\textbf{Step VI - Brute force search:}
Now, we have proved that $m \le n \le 2300$, a value which is small enough to provide a brute force search.
Namely, we used trial division with primes up to $200$ for each quantity $F_n+F_m$. We found $324$ solutions.

Finally, we note that the total computational time of our algorithm for the sequence $\{F_n\}_{n\geq 0}$ was about four and a half
hours on a computer with an Intel Core 5 M3230 processor. The most time consuming step was the first LLL-algorithm in Step IV
which took almost four hours.
\end{proof}

\section{Appendix - Explicit bounds}\label{Sec:Appendix}
Let us denote by $P=\max_{1\leq i\leq s}\{p_i\}$.

Constants appearing in Lemma \ref{lem:bounds}:
\begin{gather*}
c_1=2\frac{|a|+|b|}{\sqrt{\Delta}}, \quad c_2=\frac{\ls{\frac{c_1}{|w|}}}{\la},\quad
c_3=\max\left\{\frac{\ls{\frac{4|b|\varphi}{|a|(\varphi-1)}}}{\log{|\alpha|}},\frac{\ls{\frac{4|b|\varphi}{
|a|(\varphi-1)}}}{\log{\frac{|\alpha|}{|\be|}}}\right\},\\
c_4=\frac{|a|(\varphi-1)}{2\varphi\sd}, \quad c_5=\frac{\log{\frac{|w|}{c_4}}}{\la}
\end{gather*}

The constant $C_1(p)$ from Bugeaud and Laurent's lower bound of linear forms in two $p$-adic logarithms:
$$C_1(p)=\frac{947 p^{f_\p}}{\log^4 p}.$$

Constants appearing in Proposition \ref{prop:p-adic} and its proof:
\begin{align*}
c_6=&\max\{c_3,17.5 \log|\alpha|(\max\{\log|2a\alpha|,\log|2b\beta|\}+0.24),P^{10},e^{10}\},\\
c_9=&\frac{\log |2a\alpha|}{\log p},\\
c_{10}=&\max\{\log|2a\alpha|,\log|2b\beta|,\log p\},\\
c_{11}=&\nu_p(\log_p (\alpha/\beta))+\frac{2}{\log p}+c_9,\\
c_{8,i}=&\max\left\{{C_1(p_i)\max\{\log|\alpha|,\log p_i\}c_{10}+c_9},c_{11}\right\},\\
c_7=&\frac{\sum_{i=1}^s{c_{8,i}\log{p_i}}}{\log |\alpha|}+c_5,
\end{align*}
where $P=\max_{1\leq i \leq s}\{p_i\}$.

Constants appearing in Proposition \ref{prop:n=m} and its proof:
\begin{align*}
c_{14}&=C_1(p)h'\left(\frac ba\right)h'\left(\frac{\beta}{\alpha}+\frac{\ls|a|}{\log p}\right)\\
c_{15}&=\frac{\sum_{i=1}^s c_{14} \log p_i}{\log |\alpha|}+c_5,\\
c_{16}&=\frac{\sum_{i=1}^s {\left(\nu_{p_i}(\log_{p_i}({\alpha}/{\beta}))\log{p_i}+(2+\ls|a|) \right) }   }{\log |\alpha|} \\
c_{13}&=\max\{4 c_{15} \log(4c_{15})^2,2(c_{5}+\log c_{16}),c_3,c_2,P^{10},e^{10}\},\\
c_{12,i}&=\frac{2\log |\alpha|}{\log p_i}c_{13}.
\end{align*}
 where
 $$h'\left(\frac ab \right)\leq \max\{\log |a|,\log|b|,\log P\},\quad \text{and} \quad h'\left(
\frac{\beta}{\alpha}\right)\leq \max\{\log |\alpha|,\log|\beta|,\log P\}.$$

 The constant $C_2(n)$ from Matveev's lower bound of linear forms in complex logarithms:
$$C_2(n)=2.31 \cdot 60^{n+3} n^{4.5}.$$

Constants appearing in the proof of Lemma \ref{lem:vanishing}
\begin{align*}
 c_{22}=&\frac{\log |b/a|}{\log |\alpha/\beta|}, \\
 c_{23}=& 1.52\cdot 10^{13} \max\{0.16, \ls\max\{|a|,|b|\} \} \log|\alpha|,\\
 c_{18}=&\max\{8c_7 c_{23}\log(27 c_7 c_{23})^3,c_6\}.
\end{align*}

Constants appearing in the rest of Section \ref{Sec:n>m}:
 \begin{align*}
c_{17}=&\frac{\log\left(2\left(1+\frac{2|b|}{|a|}\right)\right)}{\log\left(\min\left\{\frac{|\alpha|}{|\beta|}
,|\alpha| \right\}\right)},\\
  c_{19}=&\frac{2C_2(s+2)\log p_1 \dots \log p_s h(w\sd{a}^{-1})\log|\alpha|
+\log\left(1+\frac{2|b|}{|a|}\right)}{\log\left(\max\left\{\frac{|\alpha|}{|\beta|},|\alpha|
\right\}\right)},\\
  c_{20}=&\max\{8c_7 c_{19}\log(27 c_7 c_{19})^3,c_{18},c_6,c_2\},\\
  c_{21,i}=&\frac{2\log|\alpha|}{\log p_i} c_{20}.
 \end{align*}

 \def\cprime{$'$}

\end{document}